\documentclass{article}
\usepackage{graphicx} 
\usepackage{amsmath,amsfonts, amsthm,amssymb,oldgerm,amssymb,amsthm}
\usepackage{esint}
\usepackage{array}
\usepackage[utf8]{inputenc}
\usepackage[T1]{fontenc}
\usepackage{epsfig}
\usepackage{relsize}
\usepackage{mathptmx}
\usepackage{float}
\usepackage{babel}
\usepackage{authblk}
\usepackage[utf8]{inputenc}
\usepackage[T1]{fontenc}
\usepackage{caption}
\usepackage{comment}
\usepackage{tikz}
\usepackage[margin=1in]{geometry}

\newtheorem{theorem}{Theorem}

\newtheorem{remark}{Remark}

\newtheorem{definition}{Definition}

\title{The Poisson's Problem on Graphs}

\author{ Diego A. Castro G.\footnote{Universidad Tecnológica de Pereira (UTP), xandercastro@utp.edu.co}}
\date{}

\date{\today}

\begin{document}

\maketitle

\begin{abstract}
    \noindent In this paper we study the problem
    \[
      \begin{cases}
           -\Delta_d u=\mu_0 \text{ in } G \\
           u=0 \text{ on } \partial G
      \end{cases}
    \]
    where, $\Delta_d$ represent the discret Laplacian, and $\mu_0$ it is a measure defined in the vertex of the graph $G=(V,E)$. Here $V$ defined the vertex of the graph, $E$ its edges and $\partial G$ its boundary. We prove that this problem has an unique solution by using an adaption of the Perron's method for the graphs by using an idea known as Balayage.
\end{abstract}


\section{Introduction}
Graph theory has played a significant role in both pure and applied mathematics, owing to the wide range of applications it has across various fields of science and engineering \cite{Gross}. A classical problem in graph theory, for example, is the so-called postman problem, in which a postal worker must deliver packages throughout a city. If we represent the delivery locations as vertices and the streets as edges, we can ask the question: What route should the postman take to visit each vertex exactly once and return to the post office? As we can see, a relatively simple scenario—and a common, practical issue—can be formulated as a mathematically rich problem within graph theory. Similarly, many problems from other disciplines can be modeled discretely using graphs \cite{Henning}.

Let us now turn to continuous problems. Suppose we wish to implement a solution method for a certain type of equation computationally. It is well known that a computer cannot handle continuous data directly; the data must first be discretized. As a simple example, consider solving Poisson’s equation over a rectangular domain. While the mathematical formulation might be clear, how can it be implemented numerically? One must first discretize the problem, perhaps using finite element methods or the discrete Fourier transform. This leads us to the question: in terms of computational resources, which method is more efficient? Once again, we find ourselves dealing with problems that can be framed in terms of graph theory.

Graph theory has proven to be an essential tool in modeling and analyzing systems governed by probabilistic or electrical laws. In probability theory, graphs are often used to represent stochastic processes. A prime example is the Markov chain, where the states of a system are represented as vertices, and transitions with associated probabilities are encoded as directed edges. This graphical representation allows for the use of combinatorial and spectral techniques to analyze long-term behavior, convergence, and stationary distributions \cite{KemenySnell, Norris}.

In electrical engineering and mathematical physics, graphs play a fundamental role in the analysis of electrical networks. Here, vertices represent connection points (nodes), and edges correspond to resistors, capacitors, or other components with assigned weights such as conductance or resistance. Kirchhoff’s laws, which govern the flow of current and voltage in electrical circuits, can be naturally expressed using the Laplacian matrix of a graph \cite{DoyleSnell, Bollobas}. This connection enables elegant solutions to classical problems like computing effective resistance between nodes, and it also supports applications in network theory, signal processing, and even machine learning.

Moreover, the interplay between probability and electrical networks becomes particularly evident in the study of random walks on graphs. There exists a deep correspondence between the behavior of a random walker and current flow in an analogous electrical network, a duality that has been leveraged in numerous theoretical and applied contexts \cite{DoyleSnell}.

\section{Preliminaries}
In this section we present the definitions and theorems necessary to prove our main result about the existence and uniqueness for the solution to the problem

\begin{equation}
    \begin{cases}
           -\Delta_d u=\mu_0 \text{ in } G \\
           u=0 \text{ on } \partial G
      \end{cases}\label{principal_problem}
\end{equation}

\begin{definition}
   A graph $G$ is a pair $(V,E)$ where $V$ a  nonempty set of point, and $E$ a multiset of the 2-elements of set points of $V$. The elements of the set $V$ are  called vertex and the elements of the set $E$ are called edges. 
\end{definition}
A graph can be represented by using lines ''edges'' connecting the vertex.
\\

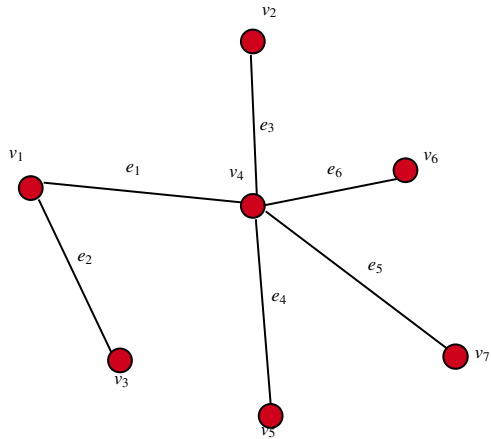
\begin{center}

\tikzset{every picture/.style={line width=0.75pt}} 

\begin{tikzpicture}[x=0.75pt,y=0.75pt,yscale=-1,xscale=1]

\draw  [fill={rgb, 255:red, 208; green, 2; blue, 27 }  ,fill opacity=1 ] (100,121) .. controls (100,117.69) and (102.69,115) .. (106,115) .. controls (109.31,115) and (112,117.69) .. (112,121) .. controls (112,124.31) and (109.31,127) .. (106,127) .. controls (102.69,127) and (100,124.31) .. (100,121) -- cycle ;
\draw  [fill={rgb, 255:red, 208; green, 2; blue, 27 }  ,fill opacity=1 ] (221,236) .. controls (221,232.69) and (223.69,230) .. (227,230) .. controls (230.31,230) and (233,232.69) .. (233,236) .. controls (233,239.31) and (230.31,242) .. (227,242) .. controls (223.69,242) and (221,239.31) .. (221,236) -- cycle ;
\draw  [fill={rgb, 255:red, 208; green, 2; blue, 27 }  ,fill opacity=1 ] (222.12,125.64) .. controls (224.53,127.91) and (224.64,131.71) .. (222.36,134.12) .. controls (220.09,136.53) and (216.29,136.64) .. (213.88,134.36) .. controls (211.47,132.09) and (211.36,128.29) .. (213.64,125.88) .. controls (215.91,123.47) and (219.71,123.36) .. (222.12,125.64) -- cycle ;
\draw  [fill={rgb, 255:red, 208; green, 2; blue, 27 }  ,fill opacity=1 ] (155.12,203.64) .. controls (157.53,205.91) and (157.64,209.71) .. (155.36,212.12) .. controls (153.09,214.53) and (149.29,214.64) .. (146.88,212.36) .. controls (144.47,210.09) and (144.36,206.29) .. (146.64,203.88) .. controls (148.91,201.47) and (152.71,201.36) .. (155.12,203.64) -- cycle ;
\draw  [color={rgb, 255:red, 0; green, 0; blue, 0 }  ,draw opacity=1 ][fill={rgb, 255:red, 208; green, 2; blue, 27 }  ,fill opacity=1 ] (299.12,107.64) .. controls (301.53,109.91) and (301.64,113.71) .. (299.36,116.12) .. controls (297.09,118.53) and (293.29,118.64) .. (290.88,116.36) .. controls (288.47,114.09) and (288.36,110.29) .. (290.64,107.88) .. controls (292.91,105.47) and (296.71,105.36) .. (299.12,107.64) -- cycle ;
\draw  [fill={rgb, 255:red, 208; green, 2; blue, 27 }  ,fill opacity=1 ] (222.12,42.64) .. controls (224.53,44.91) and (224.64,48.71) .. (222.36,51.12) .. controls (220.09,53.53) and (216.29,53.64) .. (213.88,51.36) .. controls (211.47,49.09) and (211.36,45.29) .. (213.64,42.88) .. controls (215.91,40.47) and (219.71,40.36) .. (222.12,42.64) -- cycle ;
\draw  [fill={rgb, 255:red, 208; green, 2; blue, 27 }  ,fill opacity=1 ] (316.01,210.54) .. controls (318.48,212.87) and (322.37,212.76) .. (324.7,210.29) .. controls (327.03,207.82) and (326.92,203.93) .. (324.45,201.6) .. controls (321.99,199.27) and (318.09,199.38) .. (315.76,201.85) .. controls (313.43,204.32) and (313.54,208.21) .. (316.01,210.54) -- cycle ;
\draw    (110,126.75) -- (146.64,203.88) ;
\draw    (112.5,118.25) -- (212.5,128.25) ;
\draw    (217,53.75) -- (220,124.25) ;
\draw    (219.5,136.75) -- (227,230) ;
\draw    (290.88,116.36) -- (223.5,129.75) ;
\draw    (315.76,201.85) -- (224.5,132.75) ;

\draw (93.5,100.4) node [anchor=north west][inner sep=0.75pt]  [font=\scriptsize]  {$v_{1}$};
\draw (221,27.9) node [anchor=north west][inner sep=0.75pt]  [font=\scriptsize]  {$v_{2}$};
\draw (146.5,214.4) node [anchor=north west][inner sep=0.75pt]  [font=\scriptsize]  {$v_{3}$};
\draw (204.5,109.9) node [anchor=north west][inner sep=0.75pt]  [font=\scriptsize]  {$v_{4}$};
\draw (220.5,239.9) node [anchor=north west][inner sep=0.75pt]  [font=\scriptsize]  {$v_{5}$};
\draw (302.5,101.4) node [anchor=north west][inner sep=0.75pt]  [font=\scriptsize]  {$v_{6}$};
\draw (328.5,201.4) node [anchor=north west][inner sep=0.75pt]  [font=\scriptsize]  {$v_{7}$};
\draw (152.5,107.4) node [anchor=north west][inner sep=0.75pt]  [font=\scriptsize]  {$e_{1}$};
\draw (128,152.4) node [anchor=north west][inner sep=0.75pt]  [font=\scriptsize]  {$e_{2}$};
\draw (220,85.9) node [anchor=north west][inner sep=0.75pt]  [font=\scriptsize]  {$e_{3}$};
\draw (226,172.4) node [anchor=north west][inner sep=0.75pt]  [font=\scriptsize]  {$e_{4}$};
\draw (274.5,156.9) node [anchor=north west][inner sep=0.75pt]  [font=\scriptsize]  {$e_{5}$};
\draw (254,108.4) node [anchor=north west][inner sep=0.75pt]  [font=\scriptsize]  {$e_{6}$};

\end{tikzpicture}

\end{center}

We consider finite graphs, in the sense that the set $V$ is finite, Moreover, there will be no double connections between vertices or loops. Two vertex $v, w$ is said adjacents if there exist an edge $e\in E$ such that $v$ and $w$ are connected by $e$, in this case we write $v\sim u$.

\begin{definition}(Simple graph)
    A simple graph is a graph that contains at most one edge between any two vertices and no loops on any vertex itself.
\end{definition}

\begin{definition}
   Let \( u, v \) be two vertices in \( G \). A path \( P_{uv} \) from \( u \) to \( v \) in \( V \) is a set of edges in \( E \) such that
\[
P_{uv} = \{ e_i \in E : u=v_1\sim v_2\sim ...\sim v_n=v \}
\]
We say that the graph is connected if, for every pair of vertices \( u, v \), there exists a path connecting \( u \) to \( v \) in \( G \).
\end{definition}

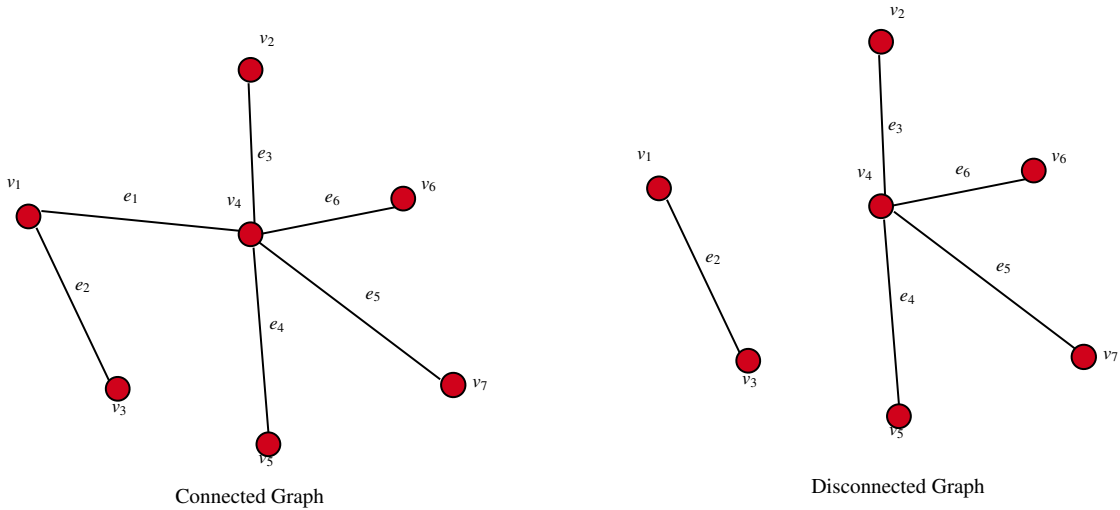
\tikzset{every picture/.style={line width=0.75pt}} 

\begin{tikzpicture}[x=0.75pt,y=0.75pt,yscale=-1,xscale=1]

\draw  [fill={rgb, 255:red, 208; green, 2; blue, 27 }  ,fill opacity=1 ] (100,121) .. controls (100,117.69) and (102.69,115) .. (106,115) .. controls (109.31,115) and (112,117.69) .. (112,121) .. controls (112,124.31) and (109.31,127) .. (106,127) .. controls (102.69,127) and (100,124.31) .. (100,121) -- cycle ;
\draw  [fill={rgb, 255:red, 208; green, 2; blue, 27 }  ,fill opacity=1 ] (221,236) .. controls (221,232.69) and (223.69,230) .. (227,230) .. controls (230.31,230) and (233,232.69) .. (233,236) .. controls (233,239.31) and (230.31,242) .. (227,242) .. controls (223.69,242) and (221,239.31) .. (221,236) -- cycle ;
\draw  [fill={rgb, 255:red, 208; green, 2; blue, 27 }  ,fill opacity=1 ] (222.12,125.64) .. controls (224.53,127.91) and (224.64,131.71) .. (222.36,134.12) .. controls (220.09,136.53) and (216.29,136.64) .. (213.88,134.36) .. controls (211.47,132.09) and (211.36,128.29) .. (213.64,125.88) .. controls (215.91,123.47) and (219.71,123.36) .. (222.12,125.64) -- cycle ;
\draw  [fill={rgb, 255:red, 208; green, 2; blue, 27 }  ,fill opacity=1 ] (155.12,203.64) .. controls (157.53,205.91) and (157.64,209.71) .. (155.36,212.12) .. controls (153.09,214.53) and (149.29,214.64) .. (146.88,212.36) .. controls (144.47,210.09) and (144.36,206.29) .. (146.64,203.88) .. controls (148.91,201.47) and (152.71,201.36) .. (155.12,203.64) -- cycle ;
\draw  [color={rgb, 255:red, 0; green, 0; blue, 0 }  ,draw opacity=1 ][fill={rgb, 255:red, 208; green, 2; blue, 27 }  ,fill opacity=1 ] (299.12,107.64) .. controls (301.53,109.91) and (301.64,113.71) .. (299.36,116.12) .. controls (297.09,118.53) and (293.29,118.64) .. (290.88,116.36) .. controls (288.47,114.09) and (288.36,110.29) .. (290.64,107.88) .. controls (292.91,105.47) and (296.71,105.36) .. (299.12,107.64) -- cycle ;
\draw  [fill={rgb, 255:red, 208; green, 2; blue, 27 }  ,fill opacity=1 ] (222.12,42.64) .. controls (224.53,44.91) and (224.64,48.71) .. (222.36,51.12) .. controls (220.09,53.53) and (216.29,53.64) .. (213.88,51.36) .. controls (211.47,49.09) and (211.36,45.29) .. (213.64,42.88) .. controls (215.91,40.47) and (219.71,40.36) .. (222.12,42.64) -- cycle ;
\draw  [fill={rgb, 255:red, 208; green, 2; blue, 27 }  ,fill opacity=1 ] (316.01,210.54) .. controls (318.48,212.87) and (322.37,212.76) .. (324.7,210.29) .. controls (327.03,207.82) and (326.92,203.93) .. (324.45,201.6) .. controls (321.99,199.27) and (318.09,199.38) .. (315.76,201.85) .. controls (313.43,204.32) and (313.54,208.21) .. (316.01,210.54) -- cycle ;
\draw    (110,126.75) -- (146.64,203.88) ;
\draw    (112.5,118.25) -- (212.5,128.25) ;
\draw    (217,53.75) -- (220,124.25) ;
\draw    (219.5,136.75) -- (227,230) ;
\draw    (290.88,116.36) -- (223.5,129.75) ;
\draw    (313.63,203.22) -- (222.36,134.12) ;
\draw  [fill={rgb, 255:red, 208; green, 2; blue, 27 }  ,fill opacity=1 ] (418,106.83) .. controls (418,103.52) and (420.69,100.83) .. (424,100.83) .. controls (427.31,100.83) and (430,103.52) .. (430,106.83) .. controls (430,110.15) and (427.31,112.83) .. (424,112.83) .. controls (420.69,112.83) and (418,110.15) .. (418,106.83) -- cycle ;
\draw  [fill={rgb, 255:red, 208; green, 2; blue, 27 }  ,fill opacity=1 ] (539,221.83) .. controls (539,218.52) and (541.69,215.83) .. (545,215.83) .. controls (548.31,215.83) and (551,218.52) .. (551,221.83) .. controls (551,225.15) and (548.31,227.83) .. (545,227.83) .. controls (541.69,227.83) and (539,225.15) .. (539,221.83) -- cycle ;
\draw  [fill={rgb, 255:red, 208; green, 2; blue, 27 }  ,fill opacity=1 ] (540.12,111.47) .. controls (542.53,113.75) and (542.64,117.54) .. (540.36,119.95) .. controls (538.09,122.36) and (534.29,122.47) .. (531.88,120.2) .. controls (529.47,117.92) and (529.36,114.12) .. (531.64,111.71) .. controls (533.91,109.3) and (537.71,109.2) .. (540.12,111.47) -- cycle ;
\draw  [fill={rgb, 255:red, 208; green, 2; blue, 27 }  ,fill opacity=1 ] (473.12,189.47) .. controls (475.53,191.75) and (475.64,195.54) .. (473.36,197.95) .. controls (471.09,200.36) and (467.29,200.47) .. (464.88,198.2) .. controls (462.47,195.92) and (462.36,192.12) .. (464.64,189.71) .. controls (466.91,187.3) and (470.71,187.2) .. (473.12,189.47) -- cycle ;
\draw  [color={rgb, 255:red, 0; green, 0; blue, 0 }  ,draw opacity=1 ][fill={rgb, 255:red, 208; green, 2; blue, 27 }  ,fill opacity=1 ] (617.12,93.47) .. controls (619.53,95.75) and (619.64,99.54) .. (617.36,101.95) .. controls (615.09,104.36) and (611.29,104.47) .. (608.88,102.2) .. controls (606.47,99.92) and (606.36,96.12) .. (608.64,93.71) .. controls (610.91,91.3) and (614.71,91.2) .. (617.12,93.47) -- cycle ;
\draw  [fill={rgb, 255:red, 208; green, 2; blue, 27 }  ,fill opacity=1 ] (540.12,28.47) .. controls (542.53,30.75) and (542.64,34.54) .. (540.36,36.95) .. controls (538.09,39.36) and (534.29,39.47) .. (531.88,37.2) .. controls (529.47,34.92) and (529.36,31.12) .. (531.64,28.71) .. controls (533.91,26.3) and (537.71,26.2) .. (540.12,28.47) -- cycle ;
\draw  [fill={rgb, 255:red, 208; green, 2; blue, 27 }  ,fill opacity=1 ] (634.01,196.37) .. controls (636.48,198.7) and (640.37,198.59) .. (642.7,196.12) .. controls (645.03,193.65) and (644.92,189.76) .. (642.45,187.43) .. controls (639.99,185.1) and (636.09,185.21) .. (633.76,187.68) .. controls (631.43,190.15) and (631.54,194.04) .. (634.01,196.37) -- cycle ;
\draw    (428,112.58) -- (464.64,189.71) ;
\draw    (535,39.58) -- (538,110.08) ;
\draw    (537.5,122.58) -- (545,215.83) ;
\draw    (608.88,102.2) -- (541.5,115.58) ;
\draw    (633.76,187.68) -- (542.5,118.58) ;

\draw (93.5,100.4) node [anchor=north west][inner sep=0.75pt]  [font=\scriptsize]  {$v_{1}$};
\draw (221,27.9) node [anchor=north west][inner sep=0.75pt]  [font=\scriptsize]  {$v_{2}$};
\draw (146.5,214.4) node [anchor=north west][inner sep=0.75pt]  [font=\scriptsize]  {$v_{3}$};
\draw (204.5,109.9) node [anchor=north west][inner sep=0.75pt]  [font=\scriptsize]  {$v_{4}$};
\draw (220.5,239.9) node [anchor=north west][inner sep=0.75pt]  [font=\scriptsize]  {$v_{5}$};
\draw (302.5,101.4) node [anchor=north west][inner sep=0.75pt]  [font=\scriptsize]  {$v_{6}$};
\draw (328.5,201.4) node [anchor=north west][inner sep=0.75pt]  [font=\scriptsize]  {$v_{7}$};
\draw (152.5,107.4) node [anchor=north west][inner sep=0.75pt]  [font=\scriptsize]  {$e_{1}$};
\draw (128,152.4) node [anchor=north west][inner sep=0.75pt]  [font=\scriptsize]  {$e_{2}$};
\draw (220,85.9) node [anchor=north west][inner sep=0.75pt]  [font=\scriptsize]  {$e_{3}$};
\draw (226,172.4) node [anchor=north west][inner sep=0.75pt]  [font=\scriptsize]  {$e_{4}$};
\draw (274.5,156.9) node [anchor=north west][inner sep=0.75pt]  [font=\scriptsize]  {$e_{5}$};
\draw (254,108.4) node [anchor=north west][inner sep=0.75pt]  [font=\scriptsize]  {$e_{6}$};
\draw (411.5,86.23) node [anchor=north west][inner sep=0.75pt]  [font=\scriptsize]  {$v_{1}$};
\draw (539,13.73) node [anchor=north west][inner sep=0.75pt]  [font=\scriptsize]  {$v_{2}$};
\draw (464.5,200.23) node [anchor=north west][inner sep=0.75pt]  [font=\scriptsize]  {$v_{3}$};
\draw (522.5,95.73) node [anchor=north west][inner sep=0.75pt]  [font=\scriptsize]  {$v_{4}$};
\draw (538.5,225.73) node [anchor=north west][inner sep=0.75pt]  [font=\scriptsize]  {$v_{5}$};
\draw (620.5,87.23) node [anchor=north west][inner sep=0.75pt]  [font=\scriptsize]  {$v_{6}$};
\draw (646.5,187.23) node [anchor=north west][inner sep=0.75pt]  [font=\scriptsize]  {$v_{7}$};
\draw (446,138.23) node [anchor=north west][inner sep=0.75pt]  [font=\scriptsize]  {$e_{2}$};
\draw (538,71.73) node [anchor=north west][inner sep=0.75pt]  [font=\scriptsize]  {$e_{3}$};
\draw (544,158.23) node [anchor=north west][inner sep=0.75pt]  [font=\scriptsize]  {$e_{4}$};
\draw (592.5,142.73) node [anchor=north west][inner sep=0.75pt]  [font=\scriptsize]  {$e_{5}$};
\draw (572,94.23) node [anchor=north west][inner sep=0.75pt]  [font=\scriptsize]  {$e_{6}$};
\draw (178.67,256.67) node [anchor=north west][inner sep=0.75pt]  [font=\footnotesize] [align=left] {Connected Graph};
\draw (499.33,252) node [anchor=north west][inner sep=0.75pt]  [font=\footnotesize] [align=left] {Disconnected Graph};

\end{tikzpicture}
\begin{definition}
   We say that the graph \( G \) is directed if for all \( u, v \in G \) with \( u \sim v \), there exists \( e \in E \) such that \( u \) is the initial point and \( v \) is the terminal point of \( e \). We use the notation \( e=e_{uv} \) to denote the edge that has initial point at \( u \) and terminal point at \( v \), or $e_-$ for the initial point and $e_+$ for the terminal point.

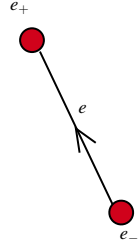
\begin{center}
    \tikzset{every picture/.style={line width=0.75pt}} 

\begin{tikzpicture}[x=0.75pt,y=0.75pt,yscale=-1,xscale=1]

\draw  [fill={rgb, 255:red, 208; green, 2; blue, 27 }  ,fill opacity=1 ] (280,375.83) .. controls (280,372.52) and (282.69,369.83) .. (286,369.83) .. controls (289.31,369.83) and (292,372.52) .. (292,375.83) .. controls (292,379.15) and (289.31,381.83) .. (286,381.83) .. controls (282.69,381.83) and (280,379.15) .. (280,375.83) -- cycle ;
\draw  [fill={rgb, 255:red, 208; green, 2; blue, 27 }  ,fill opacity=1 ] (335.12,458.47) .. controls (337.53,460.75) and (337.64,464.54) .. (335.36,466.95) .. controls (333.09,469.36) and (329.29,469.47) .. (326.88,467.2) .. controls (324.47,464.92) and (324.36,461.12) .. (326.64,458.71) .. controls (328.91,456.3) and (332.71,456.2) .. (335.12,458.47) -- cycle ;
\draw    (290,381.58) -- (326.64,458.71) ;
\draw   (310.01,432.62) -- (308.12,420.41) -- (318.15,427.63) ;

\draw (273.5,355.23) node [anchor=north west][inner sep=0.75pt]  [font=\scriptsize]  {$e_{+}$};
\draw (328.88,470.6) node [anchor=north west][inner sep=0.75pt]  [font=\scriptsize]  {$e_{-}$};
\draw (308,407.23) node [anchor=north west][inner sep=0.75pt]  [font=\scriptsize]  {$e$};

\end{tikzpicture}
\end{center}

\end{definition}

\noindent For a subset $D\subseteq V\setminus \{\infty\}$,  we define the boundary of D by $\partial D=\partial _+D \cup \partial _-D$, where
\[
\partial_{\pm}D=\{v\in V\setminus D: \text{ There exists }e\in E \text{ such that }e_{\pm}=v \text{ and } e_{\mp}\in D\}
\]
and $\nu:\partial D\rightarrow \mathbb{R}$ such that $\nu_{D}(e)=\pm 1$, if $e\in \partial _{\pm}D$ analogue to the normal vector.\\

\begin{center}
    \tikzset{every picture/.style={line width=0.75pt}} 

\begin{tikzpicture}[x=0.75pt,y=0.75pt,yscale=-1,xscale=1]

\draw  [fill={rgb, 255:red, 208; green, 2; blue, 27 }  ,fill opacity=1 ] (100,121) .. controls (100,117.69) and (102.69,115) .. (106,115) .. controls (109.31,115) and (112,117.69) .. (112,121) .. controls (112,124.31) and (109.31,127) .. (106,127) .. controls (102.69,127) and (100,124.31) .. (100,121) -- cycle ;
\draw  [fill={rgb, 255:red, 208; green, 2; blue, 27 }  ,fill opacity=1 ] (221,236) .. controls (221,232.69) and (223.69,230) .. (227,230) .. controls (230.31,230) and (233,232.69) .. (233,236) .. controls (233,239.31) and (230.31,242) .. (227,242) .. controls (223.69,242) and (221,239.31) .. (221,236) -- cycle ;
\draw  [fill={rgb, 255:red, 208; green, 2; blue, 27 }  ,fill opacity=1 ] (222.12,125.64) .. controls (224.53,127.91) and (224.64,131.71) .. (222.36,134.12) .. controls (220.09,136.53) and (216.29,136.64) .. (213.88,134.36) .. controls (211.47,132.09) and (211.36,128.29) .. (213.64,125.88) .. controls (215.91,123.47) and (219.71,123.36) .. (222.12,125.64) -- cycle ;
\draw  [fill={rgb, 255:red, 65; green, 117; blue, 5 }  ,fill opacity=1 ] (155.12,203.64) .. controls (157.53,205.91) and (157.64,209.71) .. (155.36,212.12) .. controls (153.09,214.53) and (149.29,214.64) .. (146.88,212.36) .. controls (144.47,210.09) and (144.36,206.29) .. (146.64,203.88) .. controls (148.91,201.47) and (152.71,201.36) .. (155.12,203.64) -- cycle ;
\draw  [color={rgb, 255:red, 0; green, 0; blue, 0 }  ,draw opacity=1 ][fill={rgb, 255:red, 65; green, 117; blue, 5 }  ,fill opacity=1 ] (299.12,107.64) .. controls (301.53,109.91) and (301.64,113.71) .. (299.36,116.12) .. controls (297.09,118.53) and (293.29,118.64) .. (290.88,116.36) .. controls (288.47,114.09) and (288.36,110.29) .. (290.64,107.88) .. controls (292.91,105.47) and (296.71,105.36) .. (299.12,107.64) -- cycle ;
\draw  [fill={rgb, 255:red, 65; green, 117; blue, 5 }  ,fill opacity=1 ] (222.12,42.64) .. controls (224.53,44.91) and (224.64,48.71) .. (222.36,51.12) .. controls (220.09,53.53) and (216.29,53.64) .. (213.88,51.36) .. controls (211.47,49.09) and (211.36,45.29) .. (213.64,42.88) .. controls (215.91,40.47) and (219.71,40.36) .. (222.12,42.64) -- cycle ;
\draw  [fill={rgb, 255:red, 208; green, 2; blue, 27 }  ,fill opacity=1 ] (316.01,210.54) .. controls (318.48,212.87) and (322.37,212.76) .. (324.7,210.29) .. controls (327.03,207.82) and (326.92,203.93) .. (324.45,201.6) .. controls (321.99,199.27) and (318.09,199.38) .. (315.76,201.85) .. controls (313.43,204.32) and (313.54,208.21) .. (316.01,210.54) -- cycle ;
\draw    (110,126.75) -- (146.64,203.88) ;
\draw    (112.5,118.25) -- (212.5,128.25) ;
\draw    (217,53.75) -- (220,124.25) ;
\draw    (219.5,136.75) -- (227,230) ;
\draw    (290.88,116.36) -- (223.5,129.75) ;
\draw    (313.63,203.22) -- (222.36,134.12) ;
\draw  [fill={rgb, 255:red, 65; green, 117; blue, 5 }  ,fill opacity=1 ] (54,31.83) .. controls (54,28.52) and (56.69,25.83) .. (60,25.83) .. controls (63.31,25.83) and (66,28.52) .. (66,31.83) .. controls (66,35.15) and (63.31,37.83) .. (60,37.83) .. controls (56.69,37.83) and (54,35.15) .. (54,31.83) -- cycle ;
\draw  [fill={rgb, 255:red, 208; green, 2; blue, 27 }  ,fill opacity=1 ] (109.12,114.47) .. controls (111.53,116.75) and (111.64,120.54) .. (109.36,122.95) .. controls (107.09,125.36) and (103.29,125.47) .. (100.88,123.2) .. controls (98.47,120.92) and (98.36,117.12) .. (100.64,114.71) .. controls (102.91,112.3) and (106.71,112.2) .. (109.12,114.47) -- cycle ;
\draw    (64,37.58) -- (100.64,114.71) ;
\draw  [fill={rgb, 255:red, 65; green, 117; blue, 5 }  ,fill opacity=1 ] (313.12,254.64) .. controls (315.53,256.91) and (315.64,260.71) .. (313.36,263.12) .. controls (311.09,265.53) and (307.29,265.64) .. (304.88,263.36) .. controls (302.47,261.09) and (302.36,257.29) .. (304.64,254.88) .. controls (306.91,252.47) and (310.71,252.36) .. (313.12,254.64) -- cycle ;
\draw    (233,236) -- (304.64,254.88) ;
\draw    (320.23,206.07) -- (309,259) ;

\draw (221,27.9) node [anchor=north west][inner sep=0.75pt]  [font=\scriptsize]  {$v_{2}$};
\draw (146.5,214.4) node [anchor=north west][inner sep=0.75pt]  [font=\scriptsize]  {$v_{3}$};
\draw (204.5,109.9) node [anchor=north west][inner sep=0.75pt]  [font=\scriptsize]  {$v_{4}$};
\draw (220.5,239.9) node [anchor=north west][inner sep=0.75pt]  [font=\scriptsize]  {$v_{5}$};
\draw (302.5,101.4) node [anchor=north west][inner sep=0.75pt]  [font=\scriptsize]  {$v_{6}$};
\draw (328.5,201.4) node [anchor=north west][inner sep=0.75pt]  [font=\scriptsize]  {$v_{7}$};
\draw (152.5,107.4) node [anchor=north west][inner sep=0.75pt]  [font=\scriptsize]  {$e_{1}$};
\draw (128,152.4) node [anchor=north west][inner sep=0.75pt]  [font=\scriptsize]  {$e_{2}$};
\draw (220,85.9) node [anchor=north west][inner sep=0.75pt]  [font=\scriptsize]  {$e_{3}$};
\draw (226,172.4) node [anchor=north west][inner sep=0.75pt]  [font=\scriptsize]  {$e_{4}$};
\draw (274.5,156.9) node [anchor=north west][inner sep=0.75pt]  [font=\scriptsize]  {$e_{5}$};
\draw (254,108.4) node [anchor=north west][inner sep=0.75pt]  [font=\scriptsize]  {$e_{6}$};
\draw (47.5,11.23) node [anchor=north west][inner sep=0.75pt]  [font=\scriptsize]  {$v_{1}$};
\draw (100.5,125.23) node [anchor=north west][inner sep=0.75pt]  [font=\scriptsize]  {$v_{3}$};
\draw (82,63.23) node [anchor=north west][inner sep=0.75pt]  [font=\scriptsize]  {$e_{2}$};

\end{tikzpicture}

\vspace{0.3em}
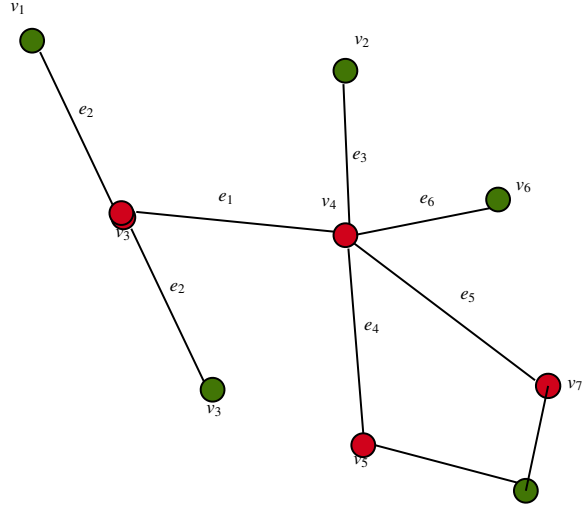
\textit{Green points represented the boundary of the set D (red points). }

\end{center}

\noindent We will introduce the basic definitions about subharmonic, superharmonic and harmonic functions and their basic properties in networks.

\begin{definition}
	Given a scalar field $U\colon V\to \mathbb{R}$, we define the gradient field by
	\[
	  DU(e)=U(e_+)-U(e_-)
	\]
\end{definition}

\begin{definition}
	Given a measure $\mu$ in $D\subseteq V$, we define the measure of $D$ with respect to $\mu$ by
	\[
	  \mu(D)=\sum \limits_{v\in D}\mu(v)
	\]
	and the integral of $U$ over $G$ with respect to the measure $\mu$ by
	\[
	  \int \limits_G Ud\mu=\sum \limits_{v\in V}U(v)\mu(v).
	\]
    
\end{definition}

\noindent Remember that for a vector field $i:\mathbb{R}^n\rightarrow \mathbb{R}^n$
\begin{align*}
\nabla\cdot i(x)&=\lim\limits_{\epsilon \rightarrow 0}\frac{1}{|B_{\epsilon}(x)|}\int \limits_{B_{\epsilon}(x)}\nabla\cdot i(y)dy\\
&=\lim\limits_{\epsilon \rightarrow 0}\frac{n}{|\partial B_{\epsilon}(x)|}\int \limits_{\partial B_{\epsilon}(x)}i(y)\cdot \nu(y)dS(y),
\end{align*}
where $\nu(y)=(y-x)/\epsilon$ is the exterior normal vector at $y\in \partial B_{\epsilon}(x)$. Thus, for $i\colon E\to \mathbb{R}$ we define
\[
\nabla\cdot i(v)=\frac{1}{|\{e\in E: v\in \{e_+,e_-\}|}\bigg(\sum \limits_{e\in E|e_-=v}i(e)-\sum \limits_{e\in E|e_+=v}i(e)\bigg).
\]

\noindent Under these constructions
\[
\nabla\cdot i(v)=\frac{1}{\# \partial\{v\}}\sum \limits_{e\in \partial \{v\}}i(e)\nu_v(e).
\]

\noindent Note that the set $\{e\in E: v\in \{e_+,e_-\}$ represented the vertices that are adjacent to vertex $v$, therefore $\# \partial\{v\}$ represent the degree of $v$.

\subsection{Divergence's Theorem}
 Remember that for a field $i: \mathbb{R}^n\rightarrow \mathbb{R}^n$ and $\Omega\subseteq \mathbb{R}^n$, if $i\in C^1(\Omega)$ and  $\Omega$ is simply connected with $\partial \Omega$ regular, then
\[
  \int \limits_{\Omega}\nabla\cdot i=\int \limits_{\partial \Omega}i\cdot \nu dS.
\]
It is known as the Divergence Theorem. We will prove an analogue of this theorem defining 
\[
\int\limits_{D}U(v)dv=\sum_{v\in D}U(v)\#\partial \{ v \}.
\]
Here the measure $dv $ meaning the degree of vertices and 

\[
\int \limits_{\partial D}i\cdot \nu dS=\sum_{ \{ e\in \partial D \} }i(e)\nu(e)
\]

\begin{theorem}[Divergence's Theorem on graphs] 
	Let $i: E\rightarrow \mathbb{R}$ and $D$, $\partial D$ as above, then
	\[
	   \int \limits_{D}\nabla\cdot i=\int \limits_{\partial D}i\cdot \nu dS
	\]
\end{theorem}
\begin{proof}
	 We will define
	\[
	  \int \limits_{D}\nabla\cdot i(v)dv=\sum \limits_{v\in D}\nabla\cdot i(v)\#\partial \{v\}.
	\]
	We will verify that if $e\in E$ such that, $e_{\pm}\in D$, then, the contribution of $i(e)$ in the integral $\int \limits_{D}\nabla\cdot i(v)$ is zero (i.e, the contribution is only by part of $e\in E$ such that if $e_{\pm}\in D$, then $e_{\mp}\in \partial D$). Let us see this. Let $v,w\in D$ such that $v\sim w$ ($v$ is adjacent to $w$) and define the edge joining $v$ with $w$ by $e_{vw}=e_{wv}$. Without loss of generality, suppose that $(e_{vw})_-=v$ and $(e_{vw})_+=w$, then 
	\begin{align*}
	 \nabla\cdot i(v)\#\partial \{v\}= & \sum \limits_{\{e\,:\,e_-=v\}}i(e)-\sum \limits_{\{e\,:\,e_+=v\}}i(e)\\
	 & = i(e_{vw})+\sum_{\substack{\{e\,:\, e_-=v,\}\\ e_+\neq w}}i(e)-\sum \limits_{\{e\,:\,e_+=v\}}i(e)
	\end{align*}
	and
	\begin{align*}
	   \nabla\cdot i(w)\#\partial \{w\}= & \sum \limits_{\{e\,:\,e_-=w\}}i(e)-\sum \limits_{\{e\,:\,e_+=w\}}i(e)\\
	  & =\sum \limits_{\{e\,:\,e_-=w\}}i(e)-\sum_{\substack{\{e\,:\, e_+=w,\}\\ e_-\neq v}}i(e)-i(e_{wv}).
	\end{align*}
	Thus, in the term  $(\nabla\cdot i(v)\#\partial \{v\}+\nabla\cdot i(w)\#\partial \{w\})$ the contribution by $i(e_{vw})=0$. This proves that the contribution is only by part of $e\in E$ such that if $e_{\pm}\in D$ then, $e_{\mp}\in \partial D$. It means that
	\begin{align*}
	\sum \limits_{x\in D}\nabla\cdot i(v)\#\partial \{v\} & =\sum_{\substack{\{e:e_-\in D,\}\\e_+\in \partial D}}i(e)-\sum_{\substack{\{e:e_+\in D,\}\\e_-\in \partial D}}i(e)\\
	&=\sum \limits_{\{e\in \partial D\}}i(e)\nu(e)\\
	&=\int  \limits_{\partial D}i\cdot \nu. 
	\end{align*}
\end{proof}

\begin{definition}
	For $i,j: E\rightarrow \mathbb{R}$ the product $i\cdot j(v)$ is defined by
	\[
	i\cdot j(v):=\frac{1}{2\#\partial \{v\}}\sum \limits_{e\in \partial \{v\}}i(e)j(e),
	\]
	where $\# \partial \{v\}$ denotes the number of elements that are adjacent to the vertex $v$ and 

      \[
      \partial \{v \}=\{ e\in E: v=e_- \text{ or } v=e_+ \}
      \]
    
\end{definition}

Now, we will prove an analogue result of the formula of integration by parts in the graph which will be very useful.

\begin{theorem}[Integration by Parts]
	Let $U:V\rightarrow \mathbb{R}$ and $i:E\rightarrow \mathbb{R}$ then, for $D\subseteq V\setminus\{\infty\}$, we have the identity
	\[
	  \int \limits_{D}U\nabla\cdot i=\int \limits_{\partial D}(\overline{U}i\cdot \nu)-\int \limits_{D}i\cdot DU.
	\]\label{parts}
    where
	\[
	  \overline{U}:E\rightarrow \mathbb{R}, \text{ with } \overline{U}(e)=\frac{1}{2}(U(e_+)+U(e_-)).
	\]
\end{theorem}

\begin{proof}
	First, it is necessary to verify the product rule
	\[
	  \nabla\cdot (\overline{U}i)(v)=U(v)\nabla\cdot i(v)+(i\cdot DU)(v)
	\]
	
	Note that
\begin{gather*}
     \nabla\cdot(\overline{U}i)(v)=\frac{1}{\# \partial \{ v \}}\bigg( \sum \limits_{e_-=v}\overline{U}(e)i(e)- \sum \limits_{e_+=v} \overline{U}(e)i(e)\bigg) \\
=\frac{1}{\#\partial \{v\}}\bigg(\sum \limits_{e_-=v}\frac{(U(v)+U(e_+))}{2}i(e)-\sum \limits_{e_+=v}\frac{(U(v)+U(e_-))}{2}i(e)\bigg)\\
 =\frac{1}{\#\partial \{v\}}\frac{U(v)}{2}\bigg(\sum \limits_{e_-=v}i(e)-\sum \limits_{e_+=v}i(e)\bigg)+\frac{1}{\#\partial \{v\}}\bigg(\sum \limits_{e_-=v}\frac{U(e_+)i(e)}{2}-\sum \limits_{e_+=v}\frac{U(e_-)i(e)}{2}\bigg)\\
 =\frac{U(v)}{2}\nabla\cdot i(v)+\frac{1}{\#\partial \{v\}}\bigg(\sum \limits_{e_-=v}\frac{U(e_+)i(e)}{2}-\sum \limits_{e_+=v}\frac{U(e_-)i(e)}{2}\bigg).
\end{gather*}
Adding and substracting the term
\[
  \frac{U(v)\nabla\cdot i(v)}{2},
\]
in the above equation, we get
\begin{align*}
\nabla\cdot(\overline{U}i)(v)&=U(v)\nabla\cdot i(v)+\frac{1}{\#\partial \{v\}}\sum \limits_{e\in \partial \{v\}}\frac{DU(e)i(e)}{2}\\
&=U(v)\nabla\cdot i(v)+(i\cdot DU)(v).
\end{align*}
Thus, applying the Gauss formula to the function $(\overline{U}i)$ we get
\begin{align*}
\int \limits_{D}\nabla\cdot (\overline{U}i)=\int \limits_{\partial D}Ui\cdot \nu
\end{align*}
and applying the product rule for $\nabla \cdot (\overline{U}i)$ we deduce
\[
  \int \limits_{D}U\nabla\cdot i=\int \limits_{\partial D} \overline{U} i\cdot \nu-\int \limits_{D}i\cdot DU.
\]
\end{proof}

\subsection{Maximum Principle.}

We will define the Laplacian in the graph and then we will prove analogues for  the mean value formula and the maximum principle.
\begin{definition}
	Let $G=(V,E)$ a graph, we define $\Delta U:V\rightarrow \mathbb{R}$ by
	\[
	  \Delta U(v)=\frac{1}{\#\partial \{v\}}\sum \limits_{w\sim v}(U(w)-U(v)).
	\]

\end{definition}
\begin{definition}
	We say that the function $U$ is subharmonic (resp. superharmonic) if $\Delta U\geq 0$ (resp. $\Delta U\leq 0$).
\end{definition}
\begin{remark}
	If $U$ be subharmonic in $D\subseteq V\setminus\{\infty\}$ then, for each $v\in D$ we have
	\[
	  U(v)\leq \frac{1}{\#\partial \{v\}}\sum \limits_{w\sim v}U(w).
	\]
\end{remark}

Now, we will present the main result in this section.
\begin{theorem}[Maximum Principle]
	Let $U$ be subharmonic in $D$. Then, $U$ attains its maximum in $\partial D$.
\end{theorem}
\begin{proof}
	Suppose that there exists $v\in D$ such that, 
	\[
	  U(v)=M=\max \limits_{\overline{D}}U
	\]
	since $U$ subharmonic, then
	\[
	  U(v)\leq \frac{1}{\#\partial \{v\}}\sum \limits_{w\sim v}U(w)
	\]
	but $U(v)\geq U(w)$ for all $w\in D$. Thus, $U(w)=M$ for $w\sim v$. Continuing in this form over any $w\sim v$ we can conclude that $U(z)=M$ for $z\sim w$ and proceeding similarly, we have that $U=M$ on $\partial D$.
\end{proof}
\begin{remark}
	If $U$ is superharmonic, the same argument for $-U$ prove that $U$ attains its minimum on $\partial D$.
\end{remark}

\begin{remark}
    Note that for us definition of Laplacian and Divergence in the graph, we have 
\begin{align*}
\nabla\cdot Du(v)&=\frac{1}{\#\partial\{v\}}\bigg(\sum \limits_{e_-=v}Du(e)-\sum \limits_{e_+=v}Du(e)\bigg)\\
&= \frac{1}{\#\partial\{v\}}\bigg(\sum \limits_{e_-=v}(u(e_+)-u(v))-\sum \limits_{e_+=v}(u(v)-u(e_-))\bigg)\\
&=\frac{1}{\#\partial \{v\}}\sum \limits_{w\sim v}(u(w)-u(v))=\Delta u(v).
\end{align*}
\end{remark}

\begin{definition}
	The fundamental solution for the Laplacian in the graph, will be the function $\Phi_x(v)$ such that, 
	\begin{equation*}
	\begin{cases}
	-\Delta \Phi_x(v)=1_{x(v)}\\
	\Phi_x(\infty)=0
	\end{cases}
	\end{equation*}
	where 
	\begin{equation*}
	1_{x(v)}=
	\begin{cases}
	1 \text{ if } v=x\\
	0 \text{ if } v\neq x
	\end{cases}
	\end{equation*}
\end{definition}
 The vertex ``$\infty$'' in the graph $G$ is any vertex that we choose, and we name it of this form, these vertex was analogue to infinity in $\mathbb{R}^n$.
 
 Now, to solve $-\Delta \Phi_x(v)=1_{x(v)}$ in the graph is equivalent to solve a system of linear equations, which has solution if, in the homogeneous problem the unique solution is the zero solution. 
 
 By the maximum principle $-\Delta \Phi_x(v)=0$ with boundary equal zero, the unique solution will be zero. This guarantee that the fundamental solution there exist.
\begin{theorem}
	The fundamental solution $\Phi_x(v)$ satisfies 
	\[
	\Phi_x(v)=\Phi_v(x) \text{ for } \, (x\neq v).
	\]
\end{theorem}
\begin{proof}
	 From the integration by parts formula, we have 
	\[
	\int \limits_{V}\Phi_x(z)\Delta \Phi_v(z)=\int \limits_{\partial V=\infty}\Phi_x(z)D\Phi_v(z)\cdot \nu-\int \limits_{V}D\Phi_x(z)\cdot D\Phi_v(z),
	\]
	but $\Phi_x(\infty)=0$ and $\Phi_v$ is harmonic for $z\neq v$, thus
	\[
	\int \limits_{V}\Phi_x(z)\Delta \Phi_v(z)=-\Phi_x(v)=-\sum \limits_{z}D\Phi_x(z)D\Phi_v(z),
	\]
	similarly
	\[
	\int \limits_{V}\Phi_v(z)\Delta \Phi_x(z)=\int \limits_{\partial V=\infty}\Phi_v(z)D\Phi_x(z)\cdot\nu-\int \limits_{V}D\Phi_v(z)\cdot D\Phi_x(z),
	\]
	thus
	\[
	-\Phi_v(x)=-\sum \limits_{z}D\Phi_x(z)D\Phi_v(z).
	\]
	Therefore
	\[
	\Phi_x(v)=\Phi_v(x).
	\]
	
\end{proof}

\newpage

\section{The Main Result}

In this section, we will present Perron's method for graphs, where a harmonic function is constructed as the lower envelope of a family of subharmonic functions. Then, in the next subsection, we will study the Poisson problem and construct its solution through a sweeping process known as balayage.

\subsection{The Perron's Method}

Let $D\subseteq V\setminus \{\infty\}$ be bounded and let $g$ defined on $\partial D$. Set
\[
S_g=\{\phi: \phi \text{ subharmonic in } D \text{ and } \phi\leq g \text{ on } \partial D\}.
\]
\begin{theorem}
	The function defined by
	\[
	u(x)=\sup \limits_{\phi \in S_g}\phi(x) 
	\]
	is harmonic in $D$.
\end{theorem}
\begin{proof}
	\hfill
	\begin{itemize}
		\item[i)] Let us see that $u$ is subharmonic.
		Fix $x_0\in D$. By definition of $u$ for all $\epsilon >0$ there exists $\phi \in S_g$ such that $u(x_0)\leq \phi(x_0)+\epsilon$, thus
		\[
		u(x_0)\leq \frac{1}{\#\partial \{x_0\}}\sum \limits_{v\sim x_0}\phi(v)+\epsilon.
		\]
		but $\phi(v)\leq u(v)$, therefore
		\[
		u(x_0)\leq\frac{1}{\#\partial \{x_0\}}\sum \limits_{v\sim x_0}u(v)+\epsilon
		\]
		and since $\epsilon$ is arbitrary we have 
		\[
		u(x_0)\leq\frac{1}{\#\partial \{x_0\}}\sum \limits_{v\sim x_0}u(v)
		\]
		implying that, $u$ is subharmonic in $D$.
		\item[ii)] Now, let us define $u^{B_{x_0}}(x)$  by
        
		\begin{equation*}
		u^{B_{x_0}}(x)=
		\begin{cases}
		\frac{1}{\#\partial \{x\}}\sum \limits_{v\sim x_0}u(v)\,\text{ if } x=x_0,\\
		u(x), \,\, x\neq x_0,
		\end{cases}
		\end{equation*}
        
		 known as the "harmonic lifting" of $u$. This definition implies that $u(x)\leq u^{B_{x_0}}(x)$ and 
		\begin{align*}
		u(x)\leq& \frac{1}{\#\partial \{x\}}\sum \limits_{v\sim x}u(v)\\
		&\leq \frac{1}{\#\partial \{x\}}\sum \limits_{v\sim x} u^{B_{x_0}}(v).
		\end{align*}
		Hence, for $x\neq x_0$ we have $u(x)= u^{B_{x_0}}(x)$, and then
		\[
		u^{B_{x_0}}(x)\leq \frac{1}{\#\partial \{x\}}\sum \limits_{v\sim x} u^{B_{x_0}}(v)
		\]
		where the equality holds when $x=x_0$. Thus, $u^{B_{x_0}}(x)$ is subharmonic in $D$, therefore, $u^{B_{x_0}}\in S_g$.

		\item[iii)] Note that $u^{B_{x_0}}(x_0)\geq u(x_0)$ but by definition of $u$, $u(x_0)\geq u^{B_{x_0}}(x_0)$ which implies that $u(x_0)=u^{B_{x_0}}(x_0)$. The latter guarantees that $u$ is harmonic in $x_0$ and since $x_0$ is arbitrary we have that $u$ is harmonic in $D$.
	\end{itemize}
\end{proof}
\begin{remark}
	Note that in the proof of Perron's method in the continuous case we had to be careful because the proof was more constructive, and we needed to prove uniform convergence using  Harnack's inequality. However, the proof in the graph is easier due to punctual and uniform convergence are equivalent.
\end{remark}

\subsection{The Poisson's Problem}

Now, we will present the main result in the Classic Balayage Process for graphs.

\begin{theorem}
	Let $D\subseteq V\setminus \{\infty\}$, $\mu_0: V\rightarrow [0,+\infty)$ such that $\text{support}\, \mu_0\subseteq D$ and $\{x_i\}$ a collection of vertices such that for every $N\in \mathbb{N}$
	\[
	  \bigcup \limits_{i\geq N}\{x_i\}=D.
	\]
	Let $u_0=0$ (the zero function) and define recursively

           \[
           \mu_i(x)=
             \begin{cases}
                 \mu_{i-1}(x)+\mu_{i-1}(x_i)\Delta \delta_{x_i}(x) \text{ if } x\sim x_i \text{ or } x=x_i\\
                 \mu_{i-1}(x) \text{ on other case.}
             \end{cases}
           \]
           where $\Delta \delta_{x_i}(x_i)=-1$ and $\Delta \delta_{x_i}(x)=1/\#\partial \{ x_i \}$ if $x \sim x_i$.
    and 

    \[
           u_i(x)=
             \begin{cases}
                 u_{i-1}(x)+\mu_{i-1}(x_i) \delta_{x_i}(x) \text{ if } x=x_i\\
                 u_{i-1}(x) \text{ if } x\neq x_i
             \end{cases}
           \]
	where $\delta_{x_i}(x)=1$ if $x=x_i$ and $\delta_{x_i}(x)=0$ if $x\neq x_i$.
    
\end{theorem}
\begin{proof}
    By the definition of $u_i$ and $\mu_i$, it is satisfied
	\begin{equation}
	\begin{cases}
	 -\Delta u_i=\mu_0-\mu_i \text{ in } D,\\
	u_i=0 \text{ on } \partial D.
	\end{cases}\label{ui}
	\end{equation}
	Now, let us see that $u_i$ converges. First, note that $u_i=u_{i-1}+\mu_{i-1}(x_i)\delta_{x_i}$ thus, $u_i\geq u_{i-1}$ implying that, $u_i$ is increasing monotone. We just need to prove that $u_i$ is bounded in $D$. To see this, since $u_i$ satisfies \eqref{ui} we define $\phi_i=\Phi^{\mu_0}-\Phi^{\mu_i}$ where $\Phi$ is the fundamental solution and $\Phi^{\mu_0}$, $\Phi^{\mu_i}$ are the Newtonian potential associated to $\mu_0$, $\mu_i$ respectively. Thus,
	\[
	  \Delta \phi_i=\mu_i-\mu_0.
	\]
	By Perron's method, we can solve
	\begin{equation*}
	\begin{cases}
	\Delta \psi_i=0 \text{ in } D,\\
	\psi_i=-\phi_i \text{ on } \partial D,
	\end{cases}
	\end{equation*}
	and $u_i$ can be written as $u_i=\phi_i+\psi_i$. Now, it is only necessary to verify that $\phi_i$ is bounded in $D$. 
	\[ 
	 \phi_i=\Phi^{\mu_i}-\Phi^{\mu_0}\leq \Phi^{\mu_0}<+\infty \text{ in } D \text{ therefore, } u_i\rightarrow u \text{ for some } u.
	\]
     This implies that $\Delta u_i\rightarrow \Delta u$ thus, $\mu_i$ converges. Note that for each $x\in D$ there exists $x_{k_1}\in \{x_i\}$ such that, $x=x_{k_1}$ and $\mu_{k_1}(x)=0$. By definition from $D$
	\[
	 D=\bigcup \limits_{i> k_1}\{x_i\},
	\] 
	therefore, there exists $x_{k_2}$ such that $x=x_{k_2}$ and $\mu_{k_2}(x)=0$ for $k_1>k_2$. Continuining in the same way we build a subsequence of $\mu_i$ such that,
	\[
	  \mu_{i_k}(x)\rightarrow 0.
	\]
	therefore, $\mu_i(x)\rightarrow 0$ and $u$ satisfies 
    
	\begin{equation*}
	\begin{cases}
	-\Delta u=\mu_0 \text{ in } D\\
	u=0 \text{ on } \partial D
	\end{cases}
	\end{equation*}
\end{proof}

\begin{remark}
   The previous theorem shows us the process through which we can construct a solution to the Poisson problem $-\Delta u = \mu_0$ in $D$ with zero boundary data. This process, known as **balayage** ("to sweep"), consists of, given a certain labeling of the vertices, $\{x_1, x_2, ..., x_m\}$, and a measure $\mu_0$ defined at each $x_n$, using the recursive measure $\mu_i(x_n) = \mu_{i-1} + \mu_{i-1}(x_n)\Delta \delta_{x_i}$ to traverse each vertex $x_i$ and distribute the "mass" associated with $\mu_0(x_n)$ equally among the neighboring vertices.

    \begin{center}

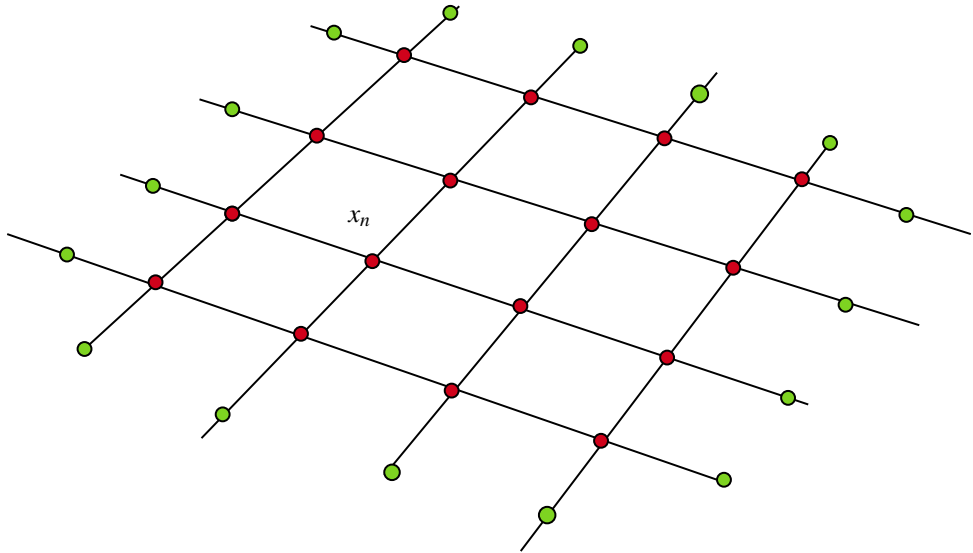
\tikzset{every picture/.style={line width=0.75pt}} 

\begin{tikzpicture}[x=0.75pt,y=0.75pt,yscale=-1,xscale=1]

\draw    (254,22.5) -- (587,127.5) ;
\draw    (198,59.5) -- (561,173.5) ;
\draw    (158,97.5) -- (505,213.5) ;
\draw    (101,127.5) -- (459,251.5) ;
\draw    (140,185.5) -- (329,12.5) ;
\draw    (199,230.5) -- (390,32.5) ;
\draw    (291.17,249.33) -- (459,46) ;
\draw    (360,287.5) -- (516,81.5) ;
\draw  [fill={rgb, 255:red, 208; green, 2; blue, 27 }  ,fill opacity=1 ] (297.67,37.17) .. controls (297.67,35.23) and (299.23,33.67) .. (301.17,33.67) .. controls (303.1,33.67) and (304.67,35.23) .. (304.67,37.17) .. controls (304.67,39.1) and (303.1,40.67) .. (301.17,40.67) .. controls (299.23,40.67) and (297.67,39.1) .. (297.67,37.17) -- cycle ;
\draw  [fill={rgb, 255:red, 126; green, 211; blue, 33 }  ,fill opacity=1 ] (512.5,81.5) .. controls (512.5,79.57) and (514.07,78) .. (516,78) .. controls (517.93,78) and (519.5,79.57) .. (519.5,81.5) .. controls (519.5,83.43) and (517.93,85) .. (516,85) .. controls (514.07,85) and (512.5,83.43) .. (512.5,81.5) -- cycle ;
\draw  [fill={rgb, 255:red, 208; green, 2; blue, 27 }  ,fill opacity=1 ] (211,117.17) .. controls (211,115.23) and (212.57,113.67) .. (214.5,113.67) .. controls (216.43,113.67) and (218,115.23) .. (218,117.17) .. controls (218,119.1) and (216.43,120.67) .. (214.5,120.67) .. controls (212.57,120.67) and (211,119.1) .. (211,117.17) -- cycle ;
\draw  [fill={rgb, 255:red, 208; green, 2; blue, 27 }  ,fill opacity=1 ] (253.67,77.83) .. controls (253.67,75.9) and (255.23,74.33) .. (257.17,74.33) .. controls (259.1,74.33) and (260.67,75.9) .. (260.67,77.83) .. controls (260.67,79.77) and (259.1,81.33) .. (257.17,81.33) .. controls (255.23,81.33) and (253.67,79.77) .. (253.67,77.83) -- cycle ;
\draw  [fill={rgb, 255:red, 208; green, 2; blue, 27 }  ,fill opacity=1 ] (321,100.5) .. controls (321,98.57) and (322.57,97) .. (324.5,97) .. controls (326.43,97) and (328,98.57) .. (328,100.5) .. controls (328,102.43) and (326.43,104) .. (324.5,104) .. controls (322.57,104) and (321,102.43) .. (321,100.5) -- cycle ;
\draw  [fill={rgb, 255:red, 208; green, 2; blue, 27 }  ,fill opacity=1 ] (281.67,141.17) .. controls (281.67,139.23) and (283.23,137.67) .. (285.17,137.67) .. controls (287.1,137.67) and (288.67,139.23) .. (288.67,141.17) .. controls (288.67,143.1) and (287.1,144.67) .. (285.17,144.67) .. controls (283.23,144.67) and (281.67,143.1) .. (281.67,141.17) -- cycle ;
\draw  [fill={rgb, 255:red, 208; green, 2; blue, 27 }  ,fill opacity=1 ] (245.67,177.83) .. controls (245.67,175.9) and (247.23,174.33) .. (249.17,174.33) .. controls (251.1,174.33) and (252.67,175.9) .. (252.67,177.83) .. controls (252.67,179.77) and (251.1,181.33) .. (249.17,181.33) .. controls (247.23,181.33) and (245.67,179.77) .. (245.67,177.83) -- cycle ;
\draw  [fill={rgb, 255:red, 208; green, 2; blue, 27 }  ,fill opacity=1 ] (172.33,151.83) .. controls (172.33,149.9) and (173.9,148.33) .. (175.83,148.33) .. controls (177.77,148.33) and (179.33,149.9) .. (179.33,151.83) .. controls (179.33,153.77) and (177.77,155.33) .. (175.83,155.33) .. controls (173.9,155.33) and (172.33,153.77) .. (172.33,151.83) -- cycle ;
\draw  [fill={rgb, 255:red, 208; green, 2; blue, 27 }  ,fill opacity=1 ] (463.67,144.5) .. controls (463.67,142.57) and (465.23,141) .. (467.17,141) .. controls (469.1,141) and (470.67,142.57) .. (470.67,144.5) .. controls (470.67,146.43) and (469.1,148) .. (467.17,148) .. controls (465.23,148) and (463.67,146.43) .. (463.67,144.5) -- cycle ;
\draw  [fill={rgb, 255:red, 208; green, 2; blue, 27 }  ,fill opacity=1 ] (430.33,189.83) .. controls (430.33,187.9) and (431.9,186.33) .. (433.83,186.33) .. controls (435.77,186.33) and (437.33,187.9) .. (437.33,189.83) .. controls (437.33,191.77) and (435.77,193.33) .. (433.83,193.33) .. controls (431.9,193.33) and (430.33,191.77) .. (430.33,189.83) -- cycle ;
\draw  [fill={rgb, 255:red, 208; green, 2; blue, 27 }  ,fill opacity=1 ] (397,231.83) .. controls (397,229.9) and (398.57,228.33) .. (400.5,228.33) .. controls (402.43,228.33) and (404,229.9) .. (404,231.83) .. controls (404,233.77) and (402.43,235.33) .. (400.5,235.33) .. controls (398.57,235.33) and (397,233.77) .. (397,231.83) -- cycle ;
\draw  [fill={rgb, 255:red, 208; green, 2; blue, 27 }  ,fill opacity=1 ] (429,79.17) .. controls (429,77.23) and (430.57,75.67) .. (432.5,75.67) .. controls (434.43,75.67) and (436,77.23) .. (436,79.17) .. controls (436,81.1) and (434.43,82.67) .. (432.5,82.67) .. controls (430.57,82.67) and (429,81.1) .. (429,79.17) -- cycle ;
\draw  [fill={rgb, 255:red, 208; green, 2; blue, 27 }  ,fill opacity=1 ] (392.33,122.5) .. controls (392.33,120.57) and (393.9,119) .. (395.83,119) .. controls (397.77,119) and (399.33,120.57) .. (399.33,122.5) .. controls (399.33,124.43) and (397.77,126) .. (395.83,126) .. controls (393.9,126) and (392.33,124.43) .. (392.33,122.5) -- cycle ;
\draw  [fill={rgb, 255:red, 208; green, 2; blue, 27 }  ,fill opacity=1 ] (356.33,163.83) .. controls (356.33,161.9) and (357.9,160.33) .. (359.83,160.33) .. controls (361.77,160.33) and (363.33,161.9) .. (363.33,163.83) .. controls (363.33,165.77) and (361.77,167.33) .. (359.83,167.33) .. controls (357.9,167.33) and (356.33,165.77) .. (356.33,163.83) -- cycle ;
\draw  [fill={rgb, 255:red, 208; green, 2; blue, 27 }  ,fill opacity=1 ] (321.67,206.5) .. controls (321.67,204.57) and (323.23,203) .. (325.17,203) .. controls (327.1,203) and (328.67,204.57) .. (328.67,206.5) .. controls (328.67,208.43) and (327.1,210) .. (325.17,210) .. controls (323.23,210) and (321.67,208.43) .. (321.67,206.5) -- cycle ;
\draw  [fill={rgb, 255:red, 208; green, 2; blue, 27 }  ,fill opacity=1 ] (361.67,58.5) .. controls (361.67,56.57) and (363.23,55) .. (365.17,55) .. controls (367.1,55) and (368.67,56.57) .. (368.67,58.5) .. controls (368.67,60.43) and (367.1,62) .. (365.17,62) .. controls (363.23,62) and (361.67,60.43) .. (361.67,58.5) -- cycle ;
\draw  [fill={rgb, 255:red, 208; green, 2; blue, 27 }  ,fill opacity=1 ] (498.33,99.83) .. controls (498.33,97.9) and (499.9,96.33) .. (501.83,96.33) .. controls (503.77,96.33) and (505.33,97.9) .. (505.33,99.83) .. controls (505.33,101.77) and (503.77,103.33) .. (501.83,103.33) .. controls (499.9,103.33) and (498.33,101.77) .. (498.33,99.83) -- cycle ;
\draw  [fill={rgb, 255:red, 126; green, 211; blue, 33 }  ,fill opacity=1 ] (446,56.75) .. controls (446,54.4) and (447.9,52.5) .. (450.25,52.5) .. controls (452.6,52.5) and (454.5,54.4) .. (454.5,56.75) .. controls (454.5,59.1) and (452.6,61) .. (450.25,61) .. controls (447.9,61) and (446,59.1) .. (446,56.75) -- cycle ;
\draw  [fill={rgb, 255:red, 126; green, 211; blue, 33 }  ,fill opacity=1 ] (386.5,32.5) .. controls (386.5,30.57) and (388.07,29) .. (390,29) .. controls (391.93,29) and (393.5,30.57) .. (393.5,32.5) .. controls (393.5,34.43) and (391.93,36) .. (390,36) .. controls (388.07,36) and (386.5,34.43) .. (386.5,32.5) -- cycle ;
\draw  [fill={rgb, 255:red, 126; green, 211; blue, 33 }  ,fill opacity=1 ] (321,15.83) .. controls (321,13.9) and (322.57,12.33) .. (324.5,12.33) .. controls (326.43,12.33) and (328,13.9) .. (328,15.83) .. controls (328,17.77) and (326.43,19.33) .. (324.5,19.33) .. controls (322.57,19.33) and (321,17.77) .. (321,15.83) -- cycle ;
\draw  [fill={rgb, 255:red, 126; green, 211; blue, 33 }  ,fill opacity=1 ] (262.33,25.83) .. controls (262.33,23.9) and (263.9,22.33) .. (265.83,22.33) .. controls (267.77,22.33) and (269.33,23.9) .. (269.33,25.83) .. controls (269.33,27.77) and (267.77,29.33) .. (265.83,29.33) .. controls (263.9,29.33) and (262.33,27.77) .. (262.33,25.83) -- cycle ;
\draw  [fill={rgb, 255:red, 126; green, 211; blue, 33 }  ,fill opacity=1 ] (211,64.5) .. controls (211,62.57) and (212.57,61) .. (214.5,61) .. controls (216.43,61) and (218,62.57) .. (218,64.5) .. controls (218,66.43) and (216.43,68) .. (214.5,68) .. controls (212.57,68) and (211,66.43) .. (211,64.5) -- cycle ;
\draw  [fill={rgb, 255:red, 126; green, 211; blue, 33 }  ,fill opacity=1 ] (171,103.17) .. controls (171,101.23) and (172.57,99.67) .. (174.5,99.67) .. controls (176.43,99.67) and (178,101.23) .. (178,103.17) .. controls (178,105.1) and (176.43,106.67) .. (174.5,106.67) .. controls (172.57,106.67) and (171,105.1) .. (171,103.17) -- cycle ;
\draw  [fill={rgb, 255:red, 126; green, 211; blue, 33 }  ,fill opacity=1 ] (127.67,137.83) .. controls (127.67,135.9) and (129.23,134.33) .. (131.17,134.33) .. controls (133.1,134.33) and (134.67,135.9) .. (134.67,137.83) .. controls (134.67,139.77) and (133.1,141.33) .. (131.17,141.33) .. controls (129.23,141.33) and (127.67,139.77) .. (127.67,137.83) -- cycle ;
\draw  [fill={rgb, 255:red, 126; green, 211; blue, 33 }  ,fill opacity=1 ] (551,117.83) .. controls (551,115.9) and (552.57,114.33) .. (554.5,114.33) .. controls (556.43,114.33) and (558,115.9) .. (558,117.83) .. controls (558,119.77) and (556.43,121.33) .. (554.5,121.33) .. controls (552.57,121.33) and (551,119.77) .. (551,117.83) -- cycle ;
\draw  [fill={rgb, 255:red, 126; green, 211; blue, 33 }  ,fill opacity=1 ] (520.33,163.17) .. controls (520.33,161.23) and (521.9,159.67) .. (523.83,159.67) .. controls (525.77,159.67) and (527.33,161.23) .. (527.33,163.17) .. controls (527.33,165.1) and (525.77,166.67) .. (523.83,166.67) .. controls (521.9,166.67) and (520.33,165.1) .. (520.33,163.17) -- cycle ;
\draw  [fill={rgb, 255:red, 126; green, 211; blue, 33 }  ,fill opacity=1 ] (491.33,210.17) .. controls (491.33,208.23) and (492.9,206.67) .. (494.83,206.67) .. controls (496.77,206.67) and (498.33,208.23) .. (498.33,210.17) .. controls (498.33,212.1) and (496.77,213.67) .. (494.83,213.67) .. controls (492.9,213.67) and (491.33,212.1) .. (491.33,210.17) -- cycle ;
\draw  [fill={rgb, 255:red, 126; green, 211; blue, 33 }  ,fill opacity=1 ] (459,251.5) .. controls (459,249.57) and (460.57,248) .. (462.5,248) .. controls (464.43,248) and (466,249.57) .. (466,251.5) .. controls (466,253.43) and (464.43,255) .. (462.5,255) .. controls (460.57,255) and (459,253.43) .. (459,251.5) -- cycle ;
\draw  [fill={rgb, 255:red, 126; green, 211; blue, 33 }  ,fill opacity=1 ] (369.17,269.33) .. controls (369.17,267.03) and (371.03,265.17) .. (373.33,265.17) .. controls (375.63,265.17) and (377.5,267.03) .. (377.5,269.33) .. controls (377.5,271.63) and (375.63,273.5) .. (373.33,273.5) .. controls (371.03,273.5) and (369.17,271.63) .. (369.17,269.33) -- cycle ;
\draw  [fill={rgb, 255:red, 126; green, 211; blue, 33 }  ,fill opacity=1 ] (299,247.75) .. controls (299,245.59) and (297.25,243.83) .. (295.08,243.83) .. controls (292.92,243.83) and (291.17,245.59) .. (291.17,247.75) .. controls (291.17,249.91) and (292.92,251.67) .. (295.08,251.67) .. controls (297.25,251.67) and (299,249.91) .. (299,247.75) -- cycle ;
\draw  [fill={rgb, 255:red, 126; green, 211; blue, 33 }  ,fill opacity=1 ] (206.17,218.5) .. controls (206.17,216.57) and (207.73,215) .. (209.67,215) .. controls (211.6,215) and (213.17,216.57) .. (213.17,218.5) .. controls (213.17,220.43) and (211.6,222) .. (209.67,222) .. controls (207.73,222) and (206.17,220.43) .. (206.17,218.5) -- cycle ;
\draw  [fill={rgb, 255:red, 126; green, 211; blue, 33 }  ,fill opacity=1 ] (136.5,185.5) .. controls (136.5,183.57) and (138.07,182) .. (140,182) .. controls (141.93,182) and (143.5,183.57) .. (143.5,185.5) .. controls (143.5,187.43) and (141.93,189) .. (140,189) .. controls (138.07,189) and (136.5,187.43) .. (136.5,185.5) -- cycle ;
\draw  [fill={rgb, 255:red, 208; green, 2; blue, 27 }  ,fill opacity=1 ] (211,117.17) .. controls (211,115.23) and (212.57,113.67) .. (214.5,113.67) .. controls (216.43,113.67) and (218,115.23) .. (218,117.17) .. controls (218,119.1) and (216.43,120.67) .. (214.5,120.67) .. controls (212.57,120.67) and (211,119.1) .. (211,117.17) -- cycle ;

\draw (271.6,114.2) node [anchor=north west][inner sep=0.75pt]    {$x_{n}$};

\end{tikzpicture}

\end{center}

Taking into account that in the $n$-th sweep we have $\mu_n(x_n) = 0$, each vertex adjacent to $x_n$ will receive $\frac{1}{\#\partial \{ x_n \}} \mu_{n-1}(x_n)$ from the swept vertex. Thus, at step $n$, the measure at $x_n$ will be zero, and it will be redistributed equally among its adjacent vertices. The idea is to continue this process until the entire set $D$ has been swept and the total mass $\mu_0(D)$ has been accumulated on the boundary $\partial D$.

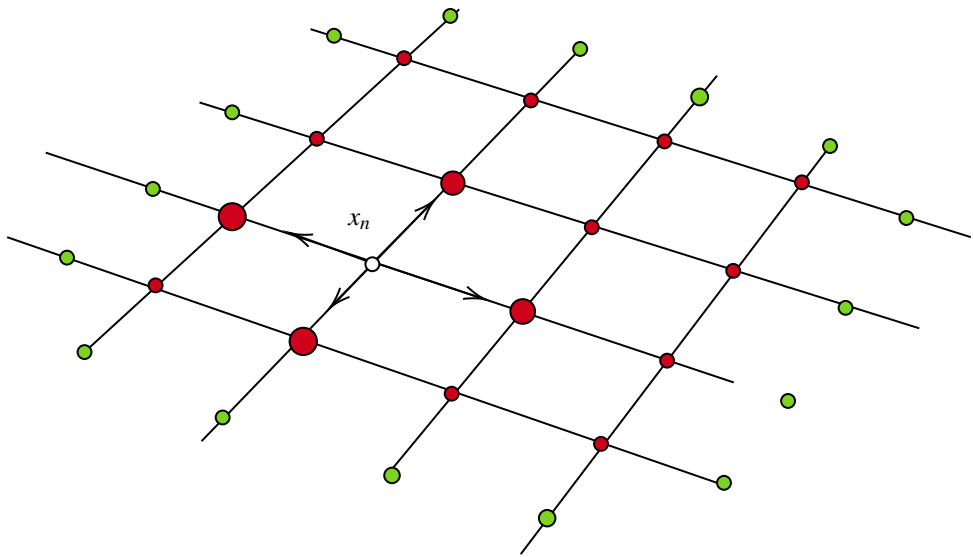
\begin{center}

\tikzset{every picture/.style={line width=0.75pt}} 

\begin{tikzpicture}[x=0.75pt,y=0.75pt,yscale=-1,xscale=1]

\draw    (254,22.5) -- (587,127.5) ;
\draw    (198,59.5) -- (561,173.5) ;
\draw    (120.44,84.83) -- (467.44,200.83) ;
\draw    (101,127.5) -- (459,251.5) ;
\draw    (140,185.5) -- (329,12.5) ;
\draw    (199,230.5) -- (390,32.5) ;
\draw    (291.17,249.33) -- (459,46) ;
\draw    (360,287.5) -- (516,81.5) ;
\draw  [fill={rgb, 255:red, 208; green, 2; blue, 27 }  ,fill opacity=1 ] (297.67,37.17) .. controls (297.67,35.23) and (299.23,33.67) .. (301.17,33.67) .. controls (303.1,33.67) and (304.67,35.23) .. (304.67,37.17) .. controls (304.67,39.1) and (303.1,40.67) .. (301.17,40.67) .. controls (299.23,40.67) and (297.67,39.1) .. (297.67,37.17) -- cycle ;
\draw  [fill={rgb, 255:red, 126; green, 211; blue, 33 }  ,fill opacity=1 ] (512.5,81.5) .. controls (512.5,79.57) and (514.07,78) .. (516,78) .. controls (517.93,78) and (519.5,79.57) .. (519.5,81.5) .. controls (519.5,83.43) and (517.93,85) .. (516,85) .. controls (514.07,85) and (512.5,83.43) .. (512.5,81.5) -- cycle ;
\draw  [fill={rgb, 255:red, 208; green, 2; blue, 27 }  ,fill opacity=1 ] (211,117.17) .. controls (211,115.23) and (212.57,113.67) .. (214.5,113.67) .. controls (216.43,113.67) and (218,115.23) .. (218,117.17) .. controls (218,119.1) and (216.43,120.67) .. (214.5,120.67) .. controls (212.57,120.67) and (211,119.1) .. (211,117.17) -- cycle ;
\draw  [fill={rgb, 255:red, 208; green, 2; blue, 27 }  ,fill opacity=1 ] (253.67,77.83) .. controls (253.67,75.9) and (255.23,74.33) .. (257.17,74.33) .. controls (259.1,74.33) and (260.67,75.9) .. (260.67,77.83) .. controls (260.67,79.77) and (259.1,81.33) .. (257.17,81.33) .. controls (255.23,81.33) and (253.67,79.77) .. (253.67,77.83) -- cycle ;
\draw  [fill={rgb, 255:red, 208; green, 2; blue, 27 }  ,fill opacity=1 ] (319.89,100.22) .. controls (319.89,96.97) and (322.53,94.33) .. (325.78,94.33) .. controls (329.03,94.33) and (331.67,96.97) .. (331.67,100.22) .. controls (331.67,103.47) and (329.03,106.11) .. (325.78,106.11) .. controls (322.53,106.11) and (319.89,103.47) .. (319.89,100.22) -- cycle ;
\draw  [fill={rgb, 255:red, 255; green, 255; blue, 255 }  ,fill opacity=1 ] (281.67,141.17) .. controls (281.67,139.23) and (283.23,137.67) .. (285.17,137.67) .. controls (287.1,137.67) and (288.67,139.23) .. (288.67,141.17) .. controls (288.67,143.1) and (287.1,144.67) .. (285.17,144.67) .. controls (283.23,144.67) and (281.67,143.1) .. (281.67,141.17) -- cycle ;
\draw  [fill={rgb, 255:red, 208; green, 2; blue, 27 }  ,fill opacity=1 ] (243.44,180.11) .. controls (243.44,176.31) and (246.53,173.22) .. (250.33,173.22) .. controls (254.14,173.22) and (257.22,176.31) .. (257.22,180.11) .. controls (257.22,183.92) and (254.14,187) .. (250.33,187) .. controls (246.53,187) and (243.44,183.92) .. (243.44,180.11) -- cycle ;
\draw  [fill={rgb, 255:red, 208; green, 2; blue, 27 }  ,fill opacity=1 ] (172.33,151.83) .. controls (172.33,149.9) and (173.9,148.33) .. (175.83,148.33) .. controls (177.77,148.33) and (179.33,149.9) .. (179.33,151.83) .. controls (179.33,153.77) and (177.77,155.33) .. (175.83,155.33) .. controls (173.9,155.33) and (172.33,153.77) .. (172.33,151.83) -- cycle ;
\draw  [fill={rgb, 255:red, 208; green, 2; blue, 27 }  ,fill opacity=1 ] (463.67,144.5) .. controls (463.67,142.57) and (465.23,141) .. (467.17,141) .. controls (469.1,141) and (470.67,142.57) .. (470.67,144.5) .. controls (470.67,146.43) and (469.1,148) .. (467.17,148) .. controls (465.23,148) and (463.67,146.43) .. (463.67,144.5) -- cycle ;
\draw  [fill={rgb, 255:red, 208; green, 2; blue, 27 }  ,fill opacity=1 ] (430.33,189.83) .. controls (430.33,187.9) and (431.9,186.33) .. (433.83,186.33) .. controls (435.77,186.33) and (437.33,187.9) .. (437.33,189.83) .. controls (437.33,191.77) and (435.77,193.33) .. (433.83,193.33) .. controls (431.9,193.33) and (430.33,191.77) .. (430.33,189.83) -- cycle ;
\draw  [fill={rgb, 255:red, 208; green, 2; blue, 27 }  ,fill opacity=1 ] (397,231.83) .. controls (397,229.9) and (398.57,228.33) .. (400.5,228.33) .. controls (402.43,228.33) and (404,229.9) .. (404,231.83) .. controls (404,233.77) and (402.43,235.33) .. (400.5,235.33) .. controls (398.57,235.33) and (397,233.77) .. (397,231.83) -- cycle ;
\draw  [fill={rgb, 255:red, 208; green, 2; blue, 27 }  ,fill opacity=1 ] (429,79.17) .. controls (429,77.23) and (430.57,75.67) .. (432.5,75.67) .. controls (434.43,75.67) and (436,77.23) .. (436,79.17) .. controls (436,81.1) and (434.43,82.67) .. (432.5,82.67) .. controls (430.57,82.67) and (429,81.1) .. (429,79.17) -- cycle ;
\draw  [fill={rgb, 255:red, 208; green, 2; blue, 27 }  ,fill opacity=1 ] (392.33,122.5) .. controls (392.33,120.57) and (393.9,119) .. (395.83,119) .. controls (397.77,119) and (399.33,120.57) .. (399.33,122.5) .. controls (399.33,124.43) and (397.77,126) .. (395.83,126) .. controls (393.9,126) and (392.33,124.43) .. (392.33,122.5) -- cycle ;
\draw  [fill={rgb, 255:red, 208; green, 2; blue, 27 }  ,fill opacity=1 ] (354.78,165) .. controls (354.78,161.56) and (357.56,158.78) .. (361,158.78) .. controls (364.44,158.78) and (367.22,161.56) .. (367.22,165) .. controls (367.22,168.44) and (364.44,171.22) .. (361,171.22) .. controls (357.56,171.22) and (354.78,168.44) .. (354.78,165) -- cycle ;
\draw  [fill={rgb, 255:red, 208; green, 2; blue, 27 }  ,fill opacity=1 ] (321.67,206.5) .. controls (321.67,204.57) and (323.23,203) .. (325.17,203) .. controls (327.1,203) and (328.67,204.57) .. (328.67,206.5) .. controls (328.67,208.43) and (327.1,210) .. (325.17,210) .. controls (323.23,210) and (321.67,208.43) .. (321.67,206.5) -- cycle ;
\draw  [fill={rgb, 255:red, 208; green, 2; blue, 27 }  ,fill opacity=1 ] (361.67,58.5) .. controls (361.67,56.57) and (363.23,55) .. (365.17,55) .. controls (367.1,55) and (368.67,56.57) .. (368.67,58.5) .. controls (368.67,60.43) and (367.1,62) .. (365.17,62) .. controls (363.23,62) and (361.67,60.43) .. (361.67,58.5) -- cycle ;
\draw  [fill={rgb, 255:red, 208; green, 2; blue, 27 }  ,fill opacity=1 ] (498.33,99.83) .. controls (498.33,97.9) and (499.9,96.33) .. (501.83,96.33) .. controls (503.77,96.33) and (505.33,97.9) .. (505.33,99.83) .. controls (505.33,101.77) and (503.77,103.33) .. (501.83,103.33) .. controls (499.9,103.33) and (498.33,101.77) .. (498.33,99.83) -- cycle ;
\draw  [fill={rgb, 255:red, 126; green, 211; blue, 33 }  ,fill opacity=1 ] (446,56.75) .. controls (446,54.4) and (447.9,52.5) .. (450.25,52.5) .. controls (452.6,52.5) and (454.5,54.4) .. (454.5,56.75) .. controls (454.5,59.1) and (452.6,61) .. (450.25,61) .. controls (447.9,61) and (446,59.1) .. (446,56.75) -- cycle ;
\draw  [fill={rgb, 255:red, 126; green, 211; blue, 33 }  ,fill opacity=1 ] (386.5,32.5) .. controls (386.5,30.57) and (388.07,29) .. (390,29) .. controls (391.93,29) and (393.5,30.57) .. (393.5,32.5) .. controls (393.5,34.43) and (391.93,36) .. (390,36) .. controls (388.07,36) and (386.5,34.43) .. (386.5,32.5) -- cycle ;
\draw  [fill={rgb, 255:red, 126; green, 211; blue, 33 }  ,fill opacity=1 ] (321,15.83) .. controls (321,13.9) and (322.57,12.33) .. (324.5,12.33) .. controls (326.43,12.33) and (328,13.9) .. (328,15.83) .. controls (328,17.77) and (326.43,19.33) .. (324.5,19.33) .. controls (322.57,19.33) and (321,17.77) .. (321,15.83) -- cycle ;
\draw  [fill={rgb, 255:red, 126; green, 211; blue, 33 }  ,fill opacity=1 ] (262.33,25.83) .. controls (262.33,23.9) and (263.9,22.33) .. (265.83,22.33) .. controls (267.77,22.33) and (269.33,23.9) .. (269.33,25.83) .. controls (269.33,27.77) and (267.77,29.33) .. (265.83,29.33) .. controls (263.9,29.33) and (262.33,27.77) .. (262.33,25.83) -- cycle ;
\draw  [fill={rgb, 255:red, 126; green, 211; blue, 33 }  ,fill opacity=1 ] (211,64.5) .. controls (211,62.57) and (212.57,61) .. (214.5,61) .. controls (216.43,61) and (218,62.57) .. (218,64.5) .. controls (218,66.43) and (216.43,68) .. (214.5,68) .. controls (212.57,68) and (211,66.43) .. (211,64.5) -- cycle ;
\draw  [fill={rgb, 255:red, 126; green, 211; blue, 33 }  ,fill opacity=1 ] (171,103.17) .. controls (171,101.23) and (172.57,99.67) .. (174.5,99.67) .. controls (176.43,99.67) and (178,101.23) .. (178,103.17) .. controls (178,105.1) and (176.43,106.67) .. (174.5,106.67) .. controls (172.57,106.67) and (171,105.1) .. (171,103.17) -- cycle ;
\draw  [fill={rgb, 255:red, 126; green, 211; blue, 33 }  ,fill opacity=1 ] (127.67,137.83) .. controls (127.67,135.9) and (129.23,134.33) .. (131.17,134.33) .. controls (133.1,134.33) and (134.67,135.9) .. (134.67,137.83) .. controls (134.67,139.77) and (133.1,141.33) .. (131.17,141.33) .. controls (129.23,141.33) and (127.67,139.77) .. (127.67,137.83) -- cycle ;
\draw  [fill={rgb, 255:red, 126; green, 211; blue, 33 }  ,fill opacity=1 ] (551,117.83) .. controls (551,115.9) and (552.57,114.33) .. (554.5,114.33) .. controls (556.43,114.33) and (558,115.9) .. (558,117.83) .. controls (558,119.77) and (556.43,121.33) .. (554.5,121.33) .. controls (552.57,121.33) and (551,119.77) .. (551,117.83) -- cycle ;
\draw  [fill={rgb, 255:red, 126; green, 211; blue, 33 }  ,fill opacity=1 ] (520.33,163.17) .. controls (520.33,161.23) and (521.9,159.67) .. (523.83,159.67) .. controls (525.77,159.67) and (527.33,161.23) .. (527.33,163.17) .. controls (527.33,165.1) and (525.77,166.67) .. (523.83,166.67) .. controls (521.9,166.67) and (520.33,165.1) .. (520.33,163.17) -- cycle ;
\draw  [fill={rgb, 255:red, 126; green, 211; blue, 33 }  ,fill opacity=1 ] (491.33,210.17) .. controls (491.33,208.23) and (492.9,206.67) .. (494.83,206.67) .. controls (496.77,206.67) and (498.33,208.23) .. (498.33,210.17) .. controls (498.33,212.1) and (496.77,213.67) .. (494.83,213.67) .. controls (492.9,213.67) and (491.33,212.1) .. (491.33,210.17) -- cycle ;
\draw  [fill={rgb, 255:red, 126; green, 211; blue, 33 }  ,fill opacity=1 ] (459,251.5) .. controls (459,249.57) and (460.57,248) .. (462.5,248) .. controls (464.43,248) and (466,249.57) .. (466,251.5) .. controls (466,253.43) and (464.43,255) .. (462.5,255) .. controls (460.57,255) and (459,253.43) .. (459,251.5) -- cycle ;
\draw  [fill={rgb, 255:red, 126; green, 211; blue, 33 }  ,fill opacity=1 ] (369.17,269.33) .. controls (369.17,267.03) and (371.03,265.17) .. (373.33,265.17) .. controls (375.63,265.17) and (377.5,267.03) .. (377.5,269.33) .. controls (377.5,271.63) and (375.63,273.5) .. (373.33,273.5) .. controls (371.03,273.5) and (369.17,271.63) .. (369.17,269.33) -- cycle ;
\draw  [fill={rgb, 255:red, 126; green, 211; blue, 33 }  ,fill opacity=1 ] (299,247.75) .. controls (299,245.59) and (297.25,243.83) .. (295.08,243.83) .. controls (292.92,243.83) and (291.17,245.59) .. (291.17,247.75) .. controls (291.17,249.91) and (292.92,251.67) .. (295.08,251.67) .. controls (297.25,251.67) and (299,249.91) .. (299,247.75) -- cycle ;
\draw  [fill={rgb, 255:red, 126; green, 211; blue, 33 }  ,fill opacity=1 ] (206.17,218.5) .. controls (206.17,216.57) and (207.73,215) .. (209.67,215) .. controls (211.6,215) and (213.17,216.57) .. (213.17,218.5) .. controls (213.17,220.43) and (211.6,222) .. (209.67,222) .. controls (207.73,222) and (206.17,220.43) .. (206.17,218.5) -- cycle ;
\draw  [fill={rgb, 255:red, 126; green, 211; blue, 33 }  ,fill opacity=1 ] (136.5,185.5) .. controls (136.5,183.57) and (138.07,182) .. (140,182) .. controls (141.93,182) and (143.5,183.57) .. (143.5,185.5) .. controls (143.5,187.43) and (141.93,189) .. (140,189) .. controls (138.07,189) and (136.5,187.43) .. (136.5,185.5) -- cycle ;
\draw  [fill={rgb, 255:red, 208; green, 2; blue, 27 }  ,fill opacity=1 ] (207.68,117.17) .. controls (207.68,113.4) and (210.73,110.35) .. (214.5,110.35) .. controls (218.27,110.35) and (221.32,113.4) .. (221.32,117.17) .. controls (221.32,120.93) and (218.27,123.99) .. (214.5,123.99) .. controls (210.73,123.99) and (207.68,120.93) .. (207.68,117.17) -- cycle ;
\draw    (287.44,138.33) -- (315.15,110.2) ;
\draw [shift={(316.56,108.78)}, rotate = 134.57] [color={rgb, 255:red, 0; green, 0; blue, 0 }  ][line width=0.75]    (10.93,-3.29) .. controls (6.95,-1.4) and (3.31,-0.3) .. (0,0) .. controls (3.31,0.3) and (6.95,1.4) .. (10.93,3.29)   ;
\draw    (288.67,141.17) -- (340.44,158.37) ;
\draw [shift={(342.33,159)}, rotate = 198.38] [color={rgb, 255:red, 0; green, 0; blue, 0 }  ][line width=0.75]    (10.93,-3.29) .. controls (6.95,-1.4) and (3.31,-0.3) .. (0,0) .. controls (3.31,0.3) and (6.95,1.4) .. (10.93,3.29)   ;
\draw    (282.56,139.67) -- (245.77,126.13) ;
\draw [shift={(243.89,125.44)}, rotate = 20.19] [color={rgb, 255:red, 0; green, 0; blue, 0 }  ][line width=0.75]    (10.93,-3.29) .. controls (6.95,-1.4) and (3.31,-0.3) .. (0,0) .. controls (3.31,0.3) and (6.95,1.4) .. (10.93,3.29)   ;
\draw    (282.56,143.67) -- (265.06,161.78) ;
\draw [shift={(263.67,163.22)}, rotate = 314.01] [color={rgb, 255:red, 0; green, 0; blue, 0 }  ][line width=0.75]    (10.93,-3.29) .. controls (6.95,-1.4) and (3.31,-0.3) .. (0,0) .. controls (3.31,0.3) and (6.95,1.4) .. (10.93,3.29)   ;

\draw (271.6,114.2) node [anchor=north west][inner sep=0.75pt]    {$x_{n}$};

\end{tikzpicture}

\end{center}

Finally, $u_n$ measures the amount of mass that has exited the set $D$ in the $n$-th iteration, which is given by $-\Delta u_i = \mu_0 - \mu_i$, where $(\mu_0 - \mu_i)(x)$ indicates the amount of mass that has been removed from $D$ in the $n$-th iteration. Therefore, since $u_i \to u$, with $-\Delta u(x) = 0$ for all $x \in D$, it follows that the set $D$ has been completely swept.

\end{remark}

\noindent \textbf{Acknowledgments:}\\
The author acknowledges financial support from the Convocatoria Interna UTP 2024, CIE 3-24-3.


\newpage

\section{Bibliography}
\begin{thebibliography}{9}
    \bibitem{Gross}J. Gross; J. Yellen; M. Anderson, Graph Theory, and its applications: Textbooks in mathematics, 2018/11/05.
    
    \bibitem{Henning} Michael A. Henning, Jan H. Van Vuuren, Graph and Network Theory; Springer Optimization and its Applications- number page, XXIX,766, Published: 04 June 2022.

   \bibitem{KemenySnell} Kemeny, J.G., Snell, J.L. (1976). Finite Markov Chains. Springer.

  \bibitem{Norris}Norris, J.R. (1998). Markov Chains. Cambridge University Press.

  \bibitem{DoyleSnell}Doyle, P.G., Snell, J.L. (1984). Random Walks and Electric Networks. Mathematical Association of America.

 \bibitem{Bollobas}Bollobás, B. (1998). Modern Graph Theory. Springer.

 \bibitem{levine} Levine, Lionel and Peres, Yuval. 2010, Scaling limits for internal aggregation models with multiple sources, Journal d’Analyse Mathématique,Springer Science and Business Media LLC, DOI={10.1007/s11854-010-0015-2},ISSN={1565-8538}.

\bibitem{Evans} L. Evans, Partial differential equations, journal: American Mathematical Society, 2010, volume 19.

\bibitem{Cafarelli} L.A. Caffarelli, The Obstacle Problem Revisited, The Journal of Fourier Analysis and Applications, 1998.

\bibitem{Gustafson}B. Gustafsson, Lectures on Balayage, In Clifford  algebras  and  potential  theory, 2004.

\bibitem{Trudinger}Gilbarg, David and Trudinger, Neil S, Elliptic partial differential equations of second order, Springer-Verlag, Berlin, 2001.




 
\end{thebibliography}
\end{document}